\documentclass[11pt, oneside]{amsart}

\usepackage{amsmath,ifthen, amsfonts, amssymb,
srcltx, 
   amsopn,tikz , url}
\usetikzlibrary{decorations.markings,arrows}

\usepackage{hyperref}

\usepackage{graphicx}

\newtheorem{thm}{Theorem}[section]
\newtheorem{lem}[thm]{Lemma}

\newtheorem{cor}[thm]{Corollary}
\newtheorem{conj}[thm]{Conjecture}
\newtheorem{prop}[thm]{Proposition}

\theoremstyle{definition}
\newtheorem{defn}[thm]{Definition}
\newtheorem{rem}[thm]{Remark}
\newtheorem{exmp}[thm]{Example}

\newtheorem{conv}[thm]{Convention}

\newtheorem{prob}[thm]{Problem}
\newtheorem{const}[thm]{Construction}

\DeclareMathOperator{\degree}{degree}

\DeclareMathOperator{\link}{link}

\DeclareMathOperator{\stab}{Stabilizer}

\newcommand{\scname}[1]{\text{\sf #1}}
\newcommand{\area}{\scname{Area}}

\newcommand{\field}[1]{\mathbb{#1}}
\newcommand{\integers}{\ensuremath{\field{Z}}}

\newcommand{\naturals}{\ensuremath{\field{N}}}

\newcommand{\boundary}   {{\ensuremath \partial}}

\newcommand{\euler}{\chi}

\newcommand{\p}{\textup{\textsf{p}}}
\newcommand{\dist}{\textup{\textsf{d}}}

\newcommand{\girth}{\text{girth}}

\newcommand{\elev}{\breve}

\newcommand{\T}{T}
\newcommand{\X}{\elev X}

\begin{document}

\title{Bicollapsibility and groups with torsion}

\author[J.~B.~Gaster]{Jonah Gaster}
\author[D.~T.~Wise]{Daniel T. Wise}
	\address{Dept. of Math. \& Stats.\\
                    McGill University \\
                    Montreal, Quebec, Canada H3A 0B9}
           \email{jbgaster@gmail.com}
           \email{wise@math.mcgill.ca}

\date{\today}
\thanks{Research supported by a CRM-ISM postdoctoral fellowship and NSERC}

\begin{abstract}
We introduce the notion of a \emph{bicollapsible} 2-complex.
This allows us to generalize the hyperbolicity of one-relator groups with torsion to a broader class of groups 
with presentations whose relators are proper powers. 
We also prove that many such groups act properly and cocompactly on a CAT(0) cube complex.
\end{abstract}

\maketitle

\section{Introduction}
A one-relator group has a presentation $\langle a,b\mid w\rangle$ with a single relator.
A surprising point in combinatorial group theory is that one-relator groups are easier
to understand when the relator is a proper power.
In particular, 
Newman proved that when $n\geq 2$, the one-relator group $\langle a,b\mid w^n\rangle$ has an easily solved word problem  \cite{Newman68,LS77}.
In modern terms,  it is a Dehn presentation for a word-hyperbolic group.

A 2-complex $X$ is \emph{bicollapsible} if for each combinatorial immersion
$Y\rightarrow X$ with $Y$ compact and 
simply-connected, either $Y$ is a tree, or a single 2-cell with a free face, or $Y$ collapses along free faces of at least two distinct cells.
This generalizes  an idea of Howie, who showed this property holds when $X$ is the 2-complex associated to a one-relator group, or more generally, a \emph{staggered} 2-complex \cite{Howie87}.
Recently, staggered 2-complexes have been generalized in two ways:
 \emph{bislim} 2-complexes \cite{HelferWise2015} and 2-complexes with a \emph{good stacking} \cite{LouderWilton2014}. 
We prove bicollapsibility for these complexes in Lemma~\ref{lem:bi bi}.
We say a presentation is bicollapsible if its associated 2-complex is.

It appears that bicollapsibility is quite common in geometric group theory.
For instance, we prove bicollapsibility for 
a nonpositively curved 2-complex in Proposition~\ref{prop:CAT0 bi},
and for a small-cancellation complex in Proposition~\ref{prop:small cancellation bicollapse}.

Following \cite{HruskaWise-Torsion},
we obtain the following generalization of
Newman's  theorem,
which is proven in the more general context of
Theorem~\ref{thm:bicollapsible branched is dehn}.

\begin{thm}\label{thm:newman generalization}
Suppose the presentation $\langle a_1,\ldots, a_r \mid w_1,\ldots, w_s\rangle$
is bicollapsible.
Then $\langle a_1,\ldots, a_r \mid w_1^{n_1},\ldots, w_s^{n_s}\rangle$
is a Dehn presentation when each $n_i\geq 2$.
\end{thm}
This theorem resolves the first half of \cite[Conjecture 5.3]{LouderWilton2018}: In view of Remark~\ref{rem:good stacking is bislim}, for any compact 2-complex that has a good stacking, the associated ``branched  2-complex'' has hyperbolic $\pi_1$.

Following the work of \cite{LauerWise07} on one-relator groups with torsion, we find that there is a naturally associated wallspace structure for bicollapsible groups with torsion:

\begin{thm}\label{thm:Wallspace bicollapsible}
Suppose the presentation $\langle a_1,\ldots, a_r \mid w_1,\ldots, w_s\rangle$
is bicollapsible and each $n_i\geq 2$.
Then the universal cover of $\langle a_1,\ldots, a_r \mid w_1^{n_1},\ldots, w_s^{n_s}\rangle$
has a wallspace structure. Each wall stabilizer is virtually free.
Moreover, if the original group is infinite,
then the new group has a codimension-1 subgroup.
\end{thm}
\noindent Recall that a subgroup $H$ of a finitely generated group $G$ is \emph{codimension-1} if the quotient of the Cayley graph of $G$ by $H$ has more than one end.

The walls we use are analogous to those employed by Lauer-Wise who cubulate one-relator groups with torsion when $n\geq 4$  \cite{LauerWise07}. 
One-relator groups with torsion were proven to act properly and cocompactly later in \cite{WiseIsraelHierarchy} using a far less natural method  depending on the Magnus hierarchy and additional methods outside the theory of one-relator groups.
We were unable to determine whether Theorem~\ref{thm:Wallspace bicollapsible} can be strengthened to assert that there is a proper action
on the associated dual cube complex.

We now describe a cubulation result that requires a slight strengthening of bicollapsibility.
Let $X$ be a 2-complex with embedded 2-cells whose boundary maps are pairwise distinct. We say $X$ is \emph{$n$-collapsing} if each compact subcomplex $Y\subset \widetilde X$ of the universal cover having $m\leq n$ cells has at least $m$ collapses along free faces. (See Definition~\ref{def:n-collapsing} for the more general case.)
We emphasize that our methods in the cubulation theorem below improve considerably on the arguments of \cite{LauerWise07} and their reach.

\begin{thm}\label{thm:3collapsing cubulated}
Suppose the presentation $\langle a_1,\ldots, a_r \mid w_1,\ldots, w_s\rangle$
is $3$-collapsing. 
Then the universal cover of the 2-complex associated to $\langle a_1,\ldots, a_r \mid w_1^{n_1},\ldots, w_s^{n_s}\rangle$
has a wallspace structure
when each $n_i\geq 2$. Moreover,
the group
$\langle a_1,\ldots, a_r \mid w_1^{n_1},\ldots, w_s^{n_s}\rangle$
acts properly and cocompactly on the associated dual CAT(0) cube complex.
\end{thm}

We describe several classes of 2-complexes that are 3-collapsing (see Proposition~\ref{prop:C(6) is 3-collapsing} and Proposition~\ref{CAT(0) is 3-collapsing}).

\section{Dehn Presentations}

A disk diagram $D$ is a compact contractible 2-complex with a chosen embedding 
$D^2\subset S^2$ in the 2-sphere.
We use the notation $\boundary_\p D$ for the \emph{boundary path} of $D$,
which can be regarded as the attaching map of a 2-cell $R_\infty$ such that
$D\cup R_\infty=S^2$.
We also use the notation $\boundary_\p R$ for the boundary path of a 2-cell.
We use the usual notation $\boundary R$ and $\boundary D$ for the topological boundary
as a subspace.

Let $X$ be a 2-complex.
A \emph{disk diagram in $X$} is a combinatorial map $D\rightarrow X$
where $D$ is a disk diagram. 
It is a classical fact, first observed by Van Kampen \cite{LS77}, that  a combinatorial path $P\rightarrow X$ is null-homotopic if and only if
there exists a disk diagram in $X$ whose boundary path is $P$,
so that the map $P\rightarrow X$ factors as $P=\boundary_\p D \rightarrow D \rightarrow X$. 
We say $D\rightarrow X$ is \emph{minimal} 
if $\area(D')\geq \area(D)$ for every disk diagram $D'\rightarrow X$ with $\boundary_\p D'=\boundary_\p D$. 
Here $\area(D)$ denotes the number of 2-cells in $D$.

A \emph{spur} in $D$ is a valence $1$ vertex in $\boundary D$.
 Note that $\boundary_\p D$ has a backtrack of the form
$ee^{-1}$ at the spur.

A \emph{shell} is a 2-cell $R$ in $D$ such that $\boundary_\p R=QS$ where
 $|Q|>|S|$ and where $Q$ is a subpath of $\boundary_\p D$.

A \emph{cutcell} is a 2-cell $R$, such that $D-\text{cl}(R)$ has more than one component. Equivalently, the preimage of $\boundary R$ in $\boundary_\p D$ consists of more than one component. [In a strong form of this notion, we require that each of these components is a nontrivial path.]

\begin{defn}
$X$ has the \emph{$[$strong$]$ generalized Dehn property}
if the following holds:
Each minimal disk diagram $D\rightarrow X$ has one of the following properties:
\begin{enumerate}
\item $D$ consists of a single 0-cell, 1-cell, or 2-cell.
\item The total number of spurs, shells, and cutcells in $D$ is at least one [two].
\end{enumerate}
\end{defn}

\begin{rem}
$X$ has the \emph{Dehn property} if each minimal disk diagram
is either a 0-cell or 2-cell, or has a spur, or has a shell.
It is a classical fact, that if $X$ is compact and has the Dehn property,
then $X$ has a linear isoperimetric function. 
See e.g.\ \cite{ABC91}.
\end{rem}

An initial version of this text contained the following: 

\begin{conj}
Let $X$ be a compact 2-complex
with the generalized Dehn property.
Then $\widetilde X$ has a linear isoperimetric function.
\end{conj}

Subsequently, Baker-Riley provided an interesting and elegant counterexample \cite{BakerRiley2019}. Nonetheless, the generalized Dehn property often yields hyperbolicity. For instance, a linear isoperimetric function is ensured under the following additional condition:

\begin{lem}
Suppose that $X$ has the generalized Dehn property, and moreover suppose that every 2-cell of $X$ has perimeter $r$. Then, for any reduced disk diagram $D$ with at least one 2-cell, we have $\area(D) \le | \partial_\p D| +1-r$.
\end{lem}

\begin{proof}
Note that if the claimed inequality holds for diagrams without spurs, then it holds for diagrams with spurs, so it suffices to prove the statement for diagrams without spurs.

We proceed by induction on $\area(D)$.
If $\area(D)=1$, then $|\partial_\p D| \ge r$, and the inequality holds.

If $D$ has a shell $R$ with $\partial_\p R= QS$, 
then let $D'=D-R-Q$, and $\partial_\p D'$ is obtained from $\partial_\p D$ by replacing $Q$ by the shorter path $S$. 
Hence the following holds, where the inductive hypothesis justifies the first inequality and $|\partial_\p D'| \le |\partial_\p D| -1$: 
\[
\area(D)=\area(D')+1\le |\partial_\p D'|+2-r \le |\partial_\p D|+1-r~.
\]

Suppose that $D$ has no spurs. In that case, $D$ has a cutcell $R$.
Let $D_1',\ldots,D_n'$ be the closures of the components of $D-cl(R)$, and let $D_i=D_i'\cup R$. Now $\area(D_i)< \area(D)$ for each $i$, so using the induction hypothesis we have
\begin{align*}
\area(D) &= 1 + \sum_i \area(D_i')  = 1-n + \sum_i \area(D_i) \le  1-n + \sum_i \left( |\partial_\p D_i| + 1-r \right) \\
&  = 1-n + |\partial_\p D| + (n-1) | \partial_\p R| + n(1-r) = | \partial_\p D| +1-r~.
\end{align*}
For the equality at the start of the second line, observe that in $\sum |\partial_\p D_i|$, the 1-cells of $R$ that lie on $\partial_\p D$ are counted $n$ times, while the 1-cells of $R$ that are not on $\partial_\p D$ are counted $n-1$ times. For the final equality, note that $|\partial_\p R|=r$ by assumption.
\end{proof}

 \begin{lem}
 \label{lem:complex dehn property}
 Let $\mathcal C = \{A\rightarrow X\}$ be a collection of maps of compact contractible complexes to $X$
 that has the following closure property:
 If $A\rightarrow X \in \mathcal C$ and $B\subset A$ is a contractible subcomplex,
 then the restriction $B\rightarrow X \in \mathcal C$.

Suppose that $\mathcal C$ has the additional property: For every $A \rightarrow X \in \mathcal C$, 
\begin{enumerate}
\item either $A$ is the closure of a 
single 0-cell or 1-cell, or a single 2-cell that collapses along a free face,
\item or $A$ has at least two 
cells, each of which is a cutcell, shell, or spur.
\end{enumerate}
Then for each $ A\rightarrow X \ \in \mathcal C$,
 either $A$ is a 0-cell or 2-cell that collapses along a free face,
or $A$ has at least two cells, each of which is a shell or a spur.
\end{lem}

A \emph{shell} $R$  in a 2-complex $A$
is a 2-cell such that $\boundary R$ contains an arc $Q$ such that 
the interior of $Q$ does not intersect any other cell (besides those in $Q\cup R$)
and the complement $S$ of $Q$ satisfies $|S|<|Q|$.

\begin{proof}

Consider a smallest counterexample $A \rightarrow X \in \mathcal C$.
If $A$ consists of a single 1-cell, then by contractibility $A$ has two spurs. Therefore, we may suppose $A$ has a cutcell, $R$.
The \emph{lobes} of $A$ are the 
complexes obtained from components of $A-\text{cl}(R)$ by adding $R$.

A lobe is \emph{outermost} if it has no cutcell.
Observe that if $A$ has a cutcell then $A$ has at least two outermost lobes.
Each of these must contain a spur or shell (distinct from that cutcell).
See Figure~\ref{fig:ShellsCutsLobesSpurs} for a schematic.
\end{proof}

\begin{figure}
  \centering
  \includegraphics[width=.4\textwidth]{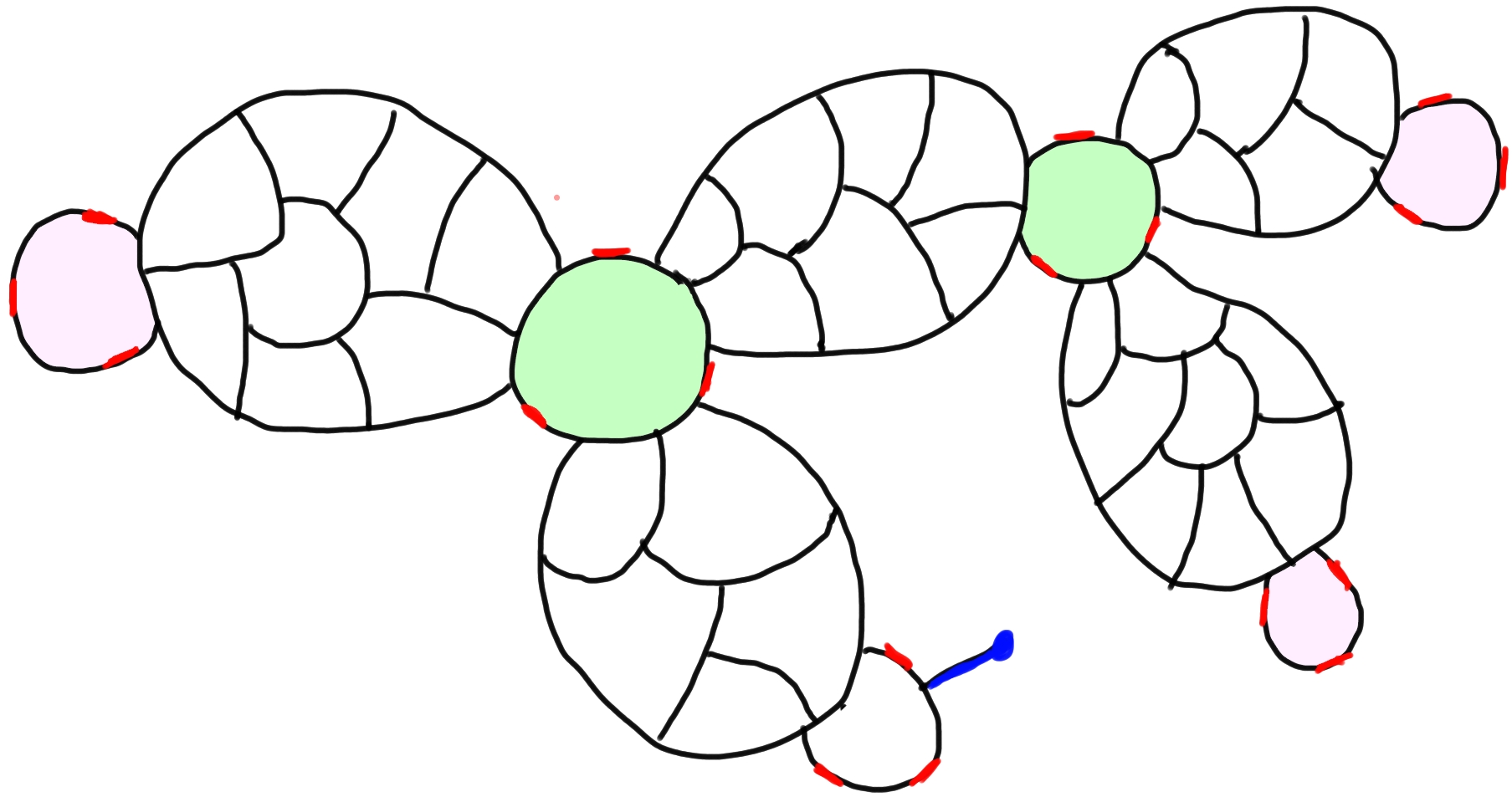}
  \caption{\label{fig:ShellsCutsLobesSpurs} 
  2 cutcells, 3 shells, and a spur.}
\end{figure}

Applying Lemma~\ref{lem:complex dehn property} to disk diagrams, we find immediately:

\begin{cor}\label{hyperbolic}
If $X$ has the strong generalized Dehn property then $X$ has
the Dehn property.
\end{cor}

\section{ Bicollapsible Complexes}
\begin{defn}
\label{def:bicollapsible}
A $2$-complex $X$ is  \emph{bicollapsible}
if for any combinatorial immersion $Y\rightarrow X$ with $Y$ compact and $\pi_1Y=1$,
one of the following holds:
\begin{enumerate}
\item
$Y$ is \emph{trivial} in the sense that $Y$ is the closure of a single 0-cell,
a single 1-cell, or single 2-cell that collapses along a free face.
\item $Y$ has two distinct cells that collapse along free faces.
\end{enumerate}
\end{defn}

\begin{rem}
It is often the case that bicollapsible arises in the following stronger form:
\emph{
If $Y\rightarrow X$ is an immersion with $\pi_1Y=1$ and $Y$ compact, then either
$Y$ consists of a single 0-cell, 1-cell, or 2-cell and is homeomorphic to a disk,
or $Y$ has at least two distinct cells that collapse along free faces.}

In particular, when all attaching maps of 2-cells are immersions,
when there is a single 2-cell but no collapse along any other free face,
the above stronger form implies that $Y^1$ is actually homeomorphic to a circle.

Examples that aren't covered by the above formulation consist of 2-complexes formed by attaching a 2-cell to a unicycle with an overly complicated attaching map so that there is still a collapse,
but so that some edges are traversed more than twice.
\end{rem}

The following provides a quick criterion for bicollapsibility.
\begin{lem}\label{lem:bicollapse criterion} Suppose each finite connected subcomplex of
the universal cover $\widetilde X$ is either a graph,
or is the closure of a single 2-cell $r$ that collapses along a free face, or  
has at least two collapses of 2-cells along free faces.
Then $X$ is bicollapsible.
\end{lem}
Note that the above condition is equivalent to $2$-collapsing, together with the assumptions that no 2-cell is a proper power and that no $2$-cells have identical attaching maps. See Definition~\ref{def:n-collapsing} for detail.
\begin{proof}For an immersion $Y\rightarrow X$ with $\pi_1Y=1$,
we consider its lift $ Y\rightarrow \widetilde X$.
Let $T=\text{image}( Y\rightarrow \widetilde X)$.
Observe that if $T$ has at least two 2-cells, then the assumption implies that $Y$ must contain two distinct cells that collapse along free faces. Thus we may assume $T$ has at most one 2-cell.

If $T$ is a graph, then either $Y$ is a single 0-cell or 1-cell,
or the immersion $Y\rightarrow T$ shows that $Y$ is a tree with at least 3 vertices,
and so $Y$ has at least two spurs.

Suppose $T$ has a single 2-cell $r$ that collapses along a free face.
If $Y$ has two 2-cells, then they both collapse.
Otherwise, $Y$ has a single $2$-cell $\hat r$ which collapses.
If $Y=\text{cl}(\hat r)$ we are done.
Otherwise, $Y$ is the union of $\hat r$ together with one or more trees,
as $\pi_1Y=1$.
Each tree provides a spur, so there are collapses of two distinct cells.\end{proof}

Consider the following variation on bicollapsibility: 
We say that $X$ is \emph{weakly bicollapsible} when
Definition~\ref{def:bicollapsible} is altered by replacing 
``$\pi_1Y=1$'' with ``$Y$ is contractible''.
We close this section by relating the two notions, together with \emph{diagrammatic reducibility} (i.e.~`DR'), which is defined in Definition~\ref{def:DR}.

\begin{lem}
The following are equivalent:
\begin{enumerate}
\item $X$ is bicollapsible.
\item $X$ is weakly bicollapsible and DR. 
\end{enumerate}
\end{lem}

\begin{proof}
(1) implies DR by Proposition~\ref{prop:corson trace Collapsibility},
and the remainder of $(1\Rightarrow 2)$ holds since contractibility implies simple-connectivity.

$(2\Rightarrow 1)$ holds since $X$ is DR and hence aspherical,
and so when $Y\rightarrow X$ with $\pi_1Y=1$ we find that $Y$ is contractible,
and hence there are two collapses.
\end{proof}

The following will be useful: 
\begin{defn}
A \emph{tower map} is a composition of inclusions of subcomplexes and covering maps.
A \emph{tower lift}   of a map $f:A\rightarrow B$,
is a map $\hat f:A\rightarrow T$ and a tower map $g:T\rightarrow B$
such that $f=g\circ \hat f$.
The tower lift is \emph{maximal} if $\hat f$ is surjective and $\pi_1$-surjective.
\end{defn}
Howie proved the following in \cite{Howie87} for combinatorial maps, but the proof generalizes to cellular maps.
In our setting, Lemma~\ref{lem:towers exist} applies, since our maps
are cellular after a subdivision.

\begin{lem}\label{lem:towers exist}
Let $A\rightarrow B$ be a cellular map of complexes with $A$ compact.
There exists a maximal tower lift $A\rightarrow T$ of $A\rightarrow B$.
\end{lem}

\begin{cor}
\label{cor:2-cells embed in Xtilde}
Suppose $X$ is compact and bicollapsible, and that each 2-cell has immersed boundary cycle. Then each 2-cell of $\widetilde X$ has embedded boundary cycle.
\end{cor}

\begin{proof}
Choose a 2-cell $R$ so that $\partial_\p R = \alpha \eta$, such that $\eta$ is a nontrivial subpath that bounds a disk diagram $D$. Moreover, assume that this choice is the smallest possible, in the sense that $\area(D)$ is smallest among all such possibilities.
We claim $R\cup_\eta D$ has a maximal tower lift to a complex $T$ so that $T$ has only one collapse, which will contradict bicollapsibility.
Indeed, we will show below that $R\cup_\eta D$ does not have any cancellable pairs, and consequently the only possible collapses of $T$ are collapses of the image of $R$ along 1-cells in the image of $\alpha$.

Suppose that $R\cup_\eta D$ has a cancellable pair. By minimality of $\area(D)$, this pair must be between $R$ and a 2-cell $R'$ of $D$ along some edge $e$. Let $\gamma'$ be the maximal subpath of $\partial_\p R'$ that is equal to a subpath $\gamma$ of $\partial_\p R$.
Express $\eta$ as $\beta\gamma\delta$, and express $\partial_\p R'$ as $\alpha' \beta' \gamma' \delta'$. 

Now Let $E=D- R'-\gamma$, and let $D'$ be the disk diagram obtained by identifying $\beta\sim \beta'$ in $E$ and $\delta\sim \delta'$. Observe that $R'$ is a 2-cell with non-embedded boundary cycle with subpath $\alpha'$ that bounds the disk diagram $D'$. However, $\area(D')<\area(D)$ by construction, violating minimality of $(R,D)$. See Figure~\ref{fig:2cellsEmbed}.

Therefore, we find that $T$ must consist of a single 2-cell that collapses. Consider the collapse of $T$; this induces collapses of $R\cup_\eta D$. However there's only one possible collapse of $R\cup_\eta D$. Therefore, we must have that the image of $D$ in $T$ is a tree. It follows that $\partial_\p R$ is not immersed, a contradiction.

\begin{figure}
\centering
\includegraphics[width=12cm]{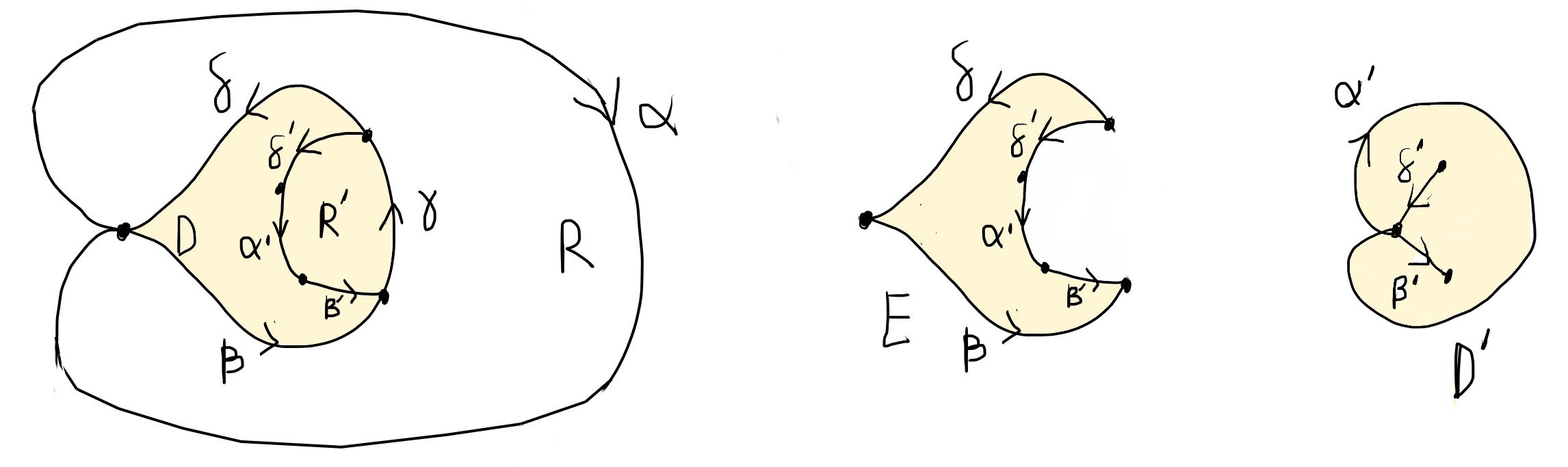}
\caption{The 2-cell $R$ with non-embedded boundary cycle.}
\label{fig:2cellsEmbed}
\end{figure}
 
\end{proof}

\section{Examples of bicollapsible complexes}
\label{examples sec}
Among various examples of bicollapsible complexes, we emphasize first that 
\emph{staggered 2-complexes without torsion} are bicollapsible by a result of Howie \cite{Howie87}.
We generalize this in \S\ref{sec:bislim bicollapse}, where we show that a bislim 2-complex is bicollapsible.

Geometric properties of $X$ can be used to ensure bicollapsibility. For instance:

\begin{prop}\label{prop:CAT0 bi}
Let $X$ be a CAT(0) 2-complex.
Then $X$ is bicollapsible.
\end{prop}
\begin{proof}
Let $Y\rightarrow X$ be an immersion with $Y$ compact and $\pi_1Y=1$. Then $Y$ is also a CAT(0) 2-complex.

Suppose that $Y$ is not a~$0$-cell. Then there exist points $p,q\in Y$ such that $\dist(p,q)$ is maximal.
The CAT(0) inequality implies that $p,q\in Y^0$.
Observe that each of $\link(p)$ and $\link(q)$ is  either
a singleton or has a spur, for otherwise the geodesic $pq$ could be extended.
Each of these cases provides a collapse.
If these two collapses are associated to the same 2-cell $r$, then either $Y=\text{cl}(r)$ and $Y$ is trivial, or
we choose $x$ to be a point of $Y$ such that $\dist(x,r)$ is maximal.
Note that $x\in Y^0$ as above, and again $\link(x)$ is either a singleton or has a spur.
This produces a second collapse, that is not a collapse of $r$.
\end{proof}

\begin{exmp}
Let $X$ be a 2-complex such that $\widetilde X$ is isomorphic to the product $T_1\times T_2$ of two trees. For instance $X=A\times B$ where $A$ and $B$ are graphs.
Then $X$ is bicollapsible by Proposition~\ref{prop:CAT0 bi}.
Usually $\euler(X)>0$, as is the case when $\euler(A),\euler(B)<0$.   
This shows that the class of bicollapsible 2-complexes goes far beyond that of 2-complexes associated to one-relator groups.
\end{exmp}

Lemma~\ref{lem:bicollapse criterion} provides a rich class of geometric examples
because of the following:

\begin{prop}\label{prop:small cancellation bicollapse}
Let $X$ be a $C(6)$ or $C(4)$-$T(4)$ complex, 
and assume distinct 2-cells of $\widetilde X$ do not have the same boundary paths.
Then $X$is bicollapsible.
\end{prop}
We refer to \cite{LS77} for a historical account and the definitions of small-cancellation theory,
and to \cite{McCammondWiseFanLadder} for a geometric treatment in line with our usage here.
Note that 2-cells in $\widetilde X$ are embedded by Greendlinger's lemma.

The proof of Proposition~\ref{prop:small cancellation bicollapse}
uses the following observation about the structure of the universal cover $\widetilde X$. See \cite[Prop.~10.2]{WiseEnergy} for detail.

\begin{figure}
  \centering
  \includegraphics[width=.7\textwidth]{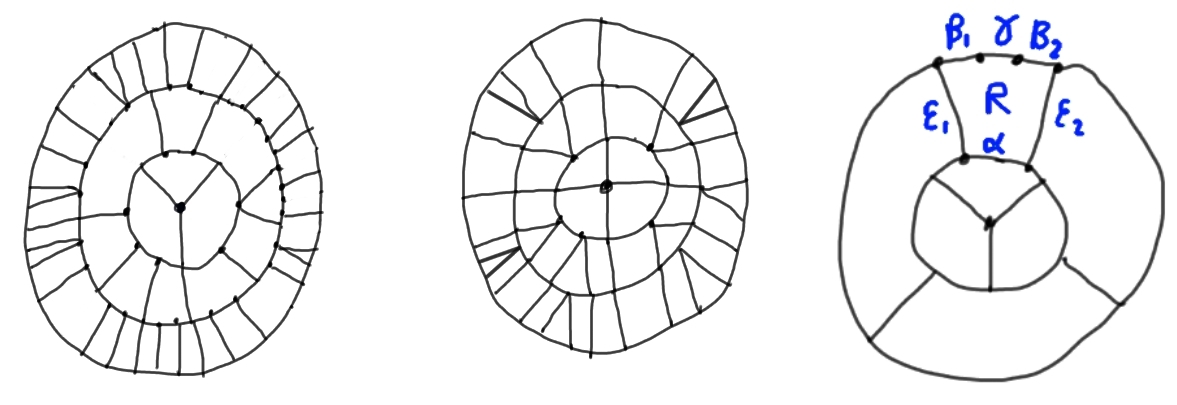}
  \caption{\label{fig:PlanarRings} $C(6)$ and $C(4)$-$T(4)$ concentric rings illustrating Lemma~\ref{lem:rings} when $\widetilde X$ is planar.}
\end{figure}

\begin{lem}[Rings]\label{lem:rings}
Let $\widetilde X$ be a simply-connected $C(6)$ $[$or $C(4)$-$T(4)$$]$ complex.
For simplicity, assume distinct  2-cells cannot have the same boundary cycle.
Then for each $0$-cell $b\in \widetilde X^0$, there is a map $g:\widetilde X \rightarrow [0,\infty)$ with the following properties:
\begin{enumerate}
\item $g(\widetilde X^0)\subset \naturals$ and $g^{-1}(0)=\{b\}$.
\item Each $1$-cell either maps to a point $n\in \naturals$ or to a segment $[n,n+1]$.
\item For each closed  $2$-cell $R$, we have $g(R)= [n,n+1]$ for some $n\in\naturals$,
and there is a decomposition  $\boundary_\p R = \alpha \varepsilon_1\beta_1\gamma\beta_2 \varepsilon_2$ where $\varepsilon_1,\varepsilon_2$ are edges,
and $g(\varepsilon_1)=g(\varepsilon_2)=[n,n+1]$ and $g(\alpha)=n$ and $g(\beta_1\gamma\beta_2)=n+1$.
\item\label{ring:abcd} $\varepsilon_1\beta_1$ and $\beta_2\varepsilon_2$ are pieces or single edges (and we allow $\beta_1$ and $\beta_2$ to be trivial)
\item \label{ring:exposure} $\alpha$ is either trivial or the concatenation of at most  two pieces $[$one piece$]$.
\item \label{ring:next} For every 2-cell $R'\neq R$, if $\boundary_\p R'$ traverses an edge of $\gamma$
then $g(R')=[n+1,n+2]$.
\end{enumerate}
\end{lem}
\begin{rem}\label{rem:nontrivial top} 
In the $C(6)$ $[$$C(4)$-$T(4)$$]$ case it follows from Conditions~\eqref{ring:abcd},~\eqref{ring:exposure}~and~\eqref{ring:next} that  $\gamma$ is 
 not the concatenation of fewer than two $[$one$]$ nontrivial pieces.
 \end{rem}
\begin{proof}[Proof of Proposition~\ref{prop:small cancellation bicollapse}]
We will apply Lemma~\ref{lem:bicollapse criterion}.
Let $Y\subset \widetilde X$ be a finite subcomplex with at least one 2-cell.
Let $b\in Y^0$.
Consider the map $g:\widetilde X\rightarrow [0,\infty)$ provided by Lemma~\ref{lem:rings} restricted to 
$g:Y\rightarrow [0,\infty)$.

Since $Y$ is finite, we may choose $n$ maximal such that $g(R)=[n,n+1]$ for some  2-cell $R$.
By Condition~\eqref{ring:abcd} $\boundary_\p R = \alpha\varepsilon_1\beta_1\gamma\beta_2\varepsilon_2$ where $\gamma$ is nontrivial
by Remark~\ref{rem:nontrivial top},
Hence for each 2-cell $R$ with $g(R)=[n,n+1]$ we have $R$ collapses along an edge of $\gamma$ by Condition~\eqref{ring:next}.

Suppose there is a unique $2$-cell with $g(R)=[n,n+1]$.
If $R$ is the only 2-cell of $Y$ then we are done, so suppose there is another 2-cell.
If there is no 2-cell $R'$ with $g(R')=[n-1,n]$ then a ``next lower maximum''  provides another collapse as above.
We thus consider a  2-cell $R'$ with $g(R')=[n-1,n]$.

In the $C(6)$ case 
we see that $\gamma' \not\subset \boundary R$ by Remark~\ref{rem:nontrivial top},
and hence $R'$  collapses along an edge of $\gamma'-\boundary R$.
Hence $Y$ has at least two collapses.

In the $C(4)$-$T(4)$ case if $R'$ has no neighboring 2-cell at $\varepsilon_1'$ and/or $\varepsilon_2'$ then we have another collapse.
So assume $R'$ has neighbors $R_1'$ and $R_2'$ (perhaps $R_1'=R_2'$).
Now observe that $\gamma_1'$ and $\gamma'$ cannot lie together  in $\boundary R$,
for then the triple $R_1', R_2', R$ would violate the $T(4)$ hypothesis.
\end{proof}

\section{Unicollapsibility and diagrammatic reducibility}
\label{sec:DR and n-collapsing}

\begin{defn}
$X$ is \emph{unicollapsible} if for every immersion
$Y\rightarrow X$ with $\pi_1Y=1$, the complex $Y$ collapses to a point,
in the sense that there is a sequence of collapses along free faces starting with $Y$ and terminating in a 0-cell.
\end{defn}

\begin{defn}[Diagrammatic Reducibility]
\label{def:DR}
A \emph{near-immersion} $A\rightarrow B$ is a combinatorial map between 2-complexes that is locally-injective except at $A^0$, and a 2-complex $X$ is \emph{diagrammatically reducible} (or DR) if there is no near-immersion $S^2\rightarrow X$ for some combinatorial 2-sphere $S^2$. 
We refer to \cite{Gersten87b}.
Note that bicollapsibility  implies unicollapsibility
which is equivalent to DR by Proposition~\ref{prop:DR Collapsibility}.
\end{defn}

The following extends the elegant Corson-Trace characterization of DR 
 \cite{CorsonTrace2000}:

\begin{prop}\label{prop:DR Collapsibility}\label{prop:corson trace Collapsibility}
The following are equivalent for the 2-complex $X$:
\begin{enumerate}
\item
 $X$ is unicollapsible.
 \item  $X$ is DR.
\item Every compact subcomplex of $\widetilde X$ collapses to a graph.
\end{enumerate}
\end{prop}

\begin{proof}
$(1\Rightarrow 2)$
Suppose  $X$ is unicollapsible.
Let $f:S^2\rightarrow X$ be a near-immersion of a 2-sphere.
By Lemma~\ref{lem:towers exist}, we have $f=p\circ \widehat f$
where
$\widehat f: S^2\rightarrow T$ is a maximal tower lift and $p:T\rightarrow X$ is a tower map.
So $T$ has a collapse along a free face $e$ of some 2-cell $r$, and this implies that $\widehat f: S^2\rightarrow T$
is not an immersion at an edge $e'$ mapping to $e$,
and hence the composition $S^2\rightarrow T \rightarrow X$ is not an immersion at $e'$.

$(2\Rightarrow 3)$
This holds by the (main part of the) Theorem of Corson and Trace in \cite{CorsonTrace2000}.

$(3\Rightarrow 1)$
Suppose $Y\rightarrow X$ is an immersion with $Y$ compact and $\pi_1Y$ trivial.
Let $\widetilde  Y$ be the image of a lift of $Y$ to $\widetilde X$.
The subcomplex $\widetilde Y$ collapses to a graph $J$.
Hence $Y$ collapses to a graph $J'$, and hence collapses to a point since $\pi_1J'=\pi_1Y=1$.
\end{proof}

\begin{prob}
Is there a characterization of bicollapsibility in the spirit of Proposition~\ref{prop:corson trace Collapsibility}?
\end{prob}

\begin{rem}Bicollapsible implies DR.
In particular, $X$ is not bicollapsible when $X$ is not aspherical.
For instance any 2-complex homeomorphic to the 2-sphere is not bicollapsible.
The simplest example of a 2-complex that isn't DR, but is still aspherical is
Zeeman's ``dunce cap'' $\langle a \mid aaa^{-1}\rangle$.

Let $X$ be the 2-complex associated to
$\langle a,b\mid ab, b\rangle$.
Then $X$ is unicollapsibile but not bicollapsible.
\end{rem}

In many natural situations a universal cover $\widetilde X$ has
 distinct 2-cells with the same boundary cycle. It will be convenient to 
 combine these ``duplicate'' 2-cells.
 
\begin{const}[quotienting duplicates]
Let $Y$ be a 2-complex.
\emph{Duplicate} 2-cells have the same boundary cycle.
Define $r:Y\rightarrow \underline Y$ to be the quotient map that identifies 
pairs of duplicate 2-cells.
Any $G$-action on $Y$ induces a $G$-action on $\underline Y$ and the map $r$ is $G$-equivariant.
We are especially interested in $r:Y\rightarrow \underline Y$ in the case that  no 2-cell of $Y$ has an attaching map that is a proper power. 
(However, it is an interesting inverse to the construction $\dot Y$ of \S\ref{sec:branched sec}, in the sense that $\underline {\dot Y} = Y$,
provided
the identification process `duplicate' is interpreted more generally so that 2-cells of $\dot Y$ project to 2-cells of $Y$ by a branched covering.)
We use the notation $\underline{\widetilde Y}$ to mean $\underline{(\widetilde Y)}$.
Of course, $\underline Y=Y$ when there are no duplicate 2-cells.
\end{const}

\begin{defn}
\label{def:n-collapsing}
Let $n\in \naturals$.
Say $X$ is \emph{$n$-collapsing}  if the following holds:
For each compact subcomplex $Y$ of $
\underline {\widetilde X}$, 
if $Y$ has at least $m\leq n$ 2-cells, then $Y$ has at least $m$ distinct collapses
along free faces.
\end{defn}

When $\widetilde X$ has no duplicate 2-cells,
 $1$-collapsing is equivalent to unicollapsing by
Proposition~\ref{prop:corson trace Collapsibility}.

\begin{lem}\label{lem:2-collapsing implies bicollapsible}
If $X$ is 2-collapsing, then $X$ is bicollapsible.
\end{lem}

\begin{proof}
If $Y\rightarrow X$ is an immersion with $\pi_1Y=1$, consider a lift $Y\rightarrow \widetilde X$. 

If the image has at least two cells, or a single 2-cell with a graph attached, we are done. 
If the image consists of a single 2-cell, this provides a collapse of each of the 2-cells of $Y$. The remaining details are left to the reader.
\end{proof}

\begin{prob}Is $2$-collapsing the same as bicollapsible?
Note that bicollapsible doesn't assert anything about subcomplexes of $\widetilde X$
that aren't simply-connected.
\end{prob}

\begin{rem}
Note that 2-collapsing implies that $\widetilde X$ has no duplicate 2-cells and no 2-cell of $X$ is attached by a proper power.
\end{rem}

\begin{rem}
The proof of Proposition~\ref{prop:CAT0 bi} actually implies that CAT(0)~$2$-complexes are~$2$-collapsing.
\end{rem}

Both Propositions~\ref{prop:CAT0 bi} and \ref{prop:small cancellation bicollapse} can be strengthened with Definition~\ref{def:n-collapsing} in mind.
\begin{prop}\label{prop:C(6) is 3-collapsing}
$C(6)$ and $C(4)$-$T(4)$ complexes are 3-collapsing.
\end{prop}
As with Proposition~\ref{prop:small cancellation bicollapse}, the proof is an application of Lemma~\ref{lem:rings} and the details are left to the reader. 
Note that $C(6)$ complexes are not always 7-collapsing and $C(4)$-$T(4)$ complexes are not always 5-collapsing. Indeed, we can surround a 2-cell by other 2-cells to see this.

In analogy with Proposition~\ref{prop:CAT0 bi}, we have:

\begin{prop}\label{CAT(0) is 3-collapsing}
Let $X$ be a CAT(0) 2-complex.
Suppose each $2$-cell is convex.
Then $X$ is $3$-collapsing.
\end{prop}
\begin{proof}
Let $Y\subset X$ be a compact subcomplex.
We first consider the case where $Y$ is planar.
Let $y$ be a point in the interior of a 2-cell.
Let $S^1$ be the space of directions about $y$,
and observe that there is a continuous map $Y-\{y\}\rightarrow S^1$ that sends each $p\in Y$ to the ray associated to the geodesic from $y$ to $p$. Note that for each point in $S^1$, there is a ray emanating from $y$
that terminates at a point on a closed 1-cell or 0-cell that collapses.
If there were only two such collapses, then this copy of $S^1$ would be contained in the union of two 2-cells of $Y$. 
By convexity, these 2-cells have connected intersection, and we find that $Y$ consists of two 2-cells that both collapse.

If $Y$ is not planar, then choose $y$ in the interior of a 1-cell $e$
such that $e$ lies on the boundary of at least three 2-cells $F_1,F_2,F_3$
Choose maximal rays $r_1,r_2,r_3$ that emanate from $y$ and are perpendicular to $e$, and begin by travelling into the $F_1,F_2,F_3$.
Let $p_1,p_2,p_3$ be the endpoints of these rays.
As in the proof of Proposition there are collapses along each $p_i$.
Suppose the collapses at the cell $P_i,P_j$ associated to $p_i,p_j$ are
the same, so $P_i=P=P_j$.
By convexity the geodesic $p_ip_j= r_i\cup r_j$ lies in $P$.
But then $P$ contains points near $y$ on ``opposite sides'' of the 1-cell $e$. This is impossible.
\end{proof}

The following observation will be strengthened in  Theorem~\ref{thm:123}.
\begin{lem}\label{lem:2cell intersections}
Let $X$ be $2$-collapsing. Then each 2-cell $R$ of $\widetilde X$ has 
$\euler(\boundary R)=0$.

Let $X$ be $3$-collapsing. Then each pair of 2-cells $R_1,R_2$ of $\widetilde X$ has
$\boundary R_1 \cap \boundary R_2$ connected.
\end{lem}

\begin{proof}
We prove the second claim, as the first is similar.
If $\boundary R_1\cap \boundary R_2$ is disconnected.
Then the subcomplex consisting of the union $Z=R_1\cup R_2$  of the closed cells,
has the property that $\pi_1\neq 1$. We may therefore choose reduced disk diagrams $D_1,\ldots, D_k$ mapping to $\widetilde X$, such that $\boundary_\p D_i$ normally generate $\pi_1Z$. Moreover choose each $D_i$ so that it is minimal area with this property.
Let $Z'= Z \cup \bigcup D_i$ be formed from $Z$ by attaching  $D_i$ along $\boundary_\p D_i$ for each $i$. Note that  $Z'$ cannot collapse along any free face except on $R_1,R_2$. Let $Y$ be the image of $Z'$ in $\widetilde X$.
Then $Y$ has at least three 2-cells, but at most two collapses.
\end{proof}

\section{Branched complexes}
\label{sec:branched sec}
\begin{conv}
We shall now work under the assumption that all 2-cells in a 2-complex $X$ have immersed boundary cycle.
\end{conv}

Let $X$ be a 2-complex with 1-skeleton $X^1$ and with 2-cells $\{C_i\}_{i\in I}$ having boundary paths $\boundary_\p C_i=w_i$. 
We form a new  2-complex $\dot X$ with the same 1-skeleton,
but whose 2-cells $\{\dot C_i\}$ have boundary paths $\boundary_p \dot C_i =w_i^{n_i}$,
where $n_i\geq 2$ for each $i\in I$.
We call $\dot X$ the \emph{branched complex} associated to $X$
and  $\{n_i\}_{i\in I}$.

\begin{lem}\label{lem:bicollapse dehn}
Let $\dot X$ be a branched complex associated to a 2-complex $X$. 
If $X$ is bicollapsible then $\dot X$ has the strong generalized Dehn property.
\end{lem}
Likewise, if $X$ is unicollapsible then
$\dot X$ has the generalized Dehn property.
\begin{proof}
Let $D\rightarrow \dot X$ be a reduced disk diagram,
and  by composing with $\dot X\rightarrow X$ we have a map
 $D\rightarrow X$.
Lemma \ref{lem:towers exist} guarantees the existence of a maximal
tower lift $D\rightarrow T$ of $D\rightarrow X$ where $T\rightarrow  X$ 
is a tower map. Note that $T$ is compact and $\pi_1T=1$.
Since $X$ is bicollapsible, either $T$ is trivial, or $T$ contains two distinct cells $C_1,C_2$ that collapse along free faces.
If $T$ is a single 0- or 1-cell then so is $D$. If $T$ consists of a single 2-cell, then $D$ is a union of 2-cells joined along vertices, and hence the claim holds.
Finally, if $T$ contains two distinct cells $C_1,C_2$ that collapse along free faces, then $D$ contains at least two spurs and/or shells and/or cutcells.
(While each collapse of a 1-cell in $T$ induces a collapse of one or more 1-cells in $D$, a collapse of a 2-cell in $T$ induces one or more shells and/or cutcells in $D$.)
\end{proof}

\begin{thm}\label{thm:bicollapsible branched is dehn} Let $X$ be compact and bicollapsible.
Then $\dot X$ has the Dehn property. Hence $\pi_1\dot X$ is word-hyperbolic.
\end{thm}
\begin{proof}
This holds by combining Lemma~\ref{lem:bicollapse dehn}
and Corollary~\ref{hyperbolic}.
\end{proof}

\begin{rem} [\'a la Newman]
Let $\dot C$ be a 2-cell of $D$ that maps to a 2-cell $C$ with $\boundary_\p C = Se$
and $C$ collapses along the free face $e$.
Suppose $\boundary_p \dot C$ has the form $(Se)^n$.
Then all $n$ copies of $e$ in $\boundary_\p \dot C$ lie on $\boundary D$.
Hence, if $\dot C$ is a shell, we see that its innerpath is a subpath of $S$.
And similarly, if $\dot C$ is a cutcell, we see that the lobes associated to $\dot C$
meet $\dot C$ along copies of $S$.
\end{rem}

We define $\elev{X}= \underline{\left(\widetilde{\dot X}\right)}.$
In fact, the map $\elev{X}\rightarrow X$ is a simply-connected branched covering space whose branching degrees are ensured to be the $\{n_i\}$:

\begin{cor}
\label{cor:orders of relators}
With the setup of Theorem~\ref{thm:bicollapsible branched is dehn}, the order of $w_i$ in $\pi_1 \dot X$ is $n_i$.
\end{cor}

\begin{proof}
Towards contradiction, suppose that $w_i^m=1$, for some $0<m<n_i$. By Theorem~\ref{thm:bicollapsible branched is dehn}, a reduced disk diagram $D$ with $\partial_\p D = w_i^m$ has a shell $w_j^{n_j}$. 
As $n_j\ge 2$, we find that a cyclic permutation of $w_i^m$ contains $w_j$ as a subword.
We may thus form a complex by gluing together copies of the corresponding 2-cells $C_i$ and $C_j$, and then fold to form an immersed complex $Y\rightarrow X$. Note that no edge on $\partial C_j$ is a free face. This contradicts bicollapsibility.
\end{proof}

\begin{rem}
The assumption that $n_i\ge 2$ in Lemma~\ref{lem:bicollapse dehn} is necessary.
For instance, the 2-complex $X$ associated to $\langle a,b\mid [a,b]\rangle$ is bicollapsible but since $\integers^2$ is not word-hyperbolic, $X$ certainly fails to have the
(strong generalized) Dehn property. 
\end{rem}

Bicollapsibility of $X$ provides control over contractible subcomplexes of $\X$.

\begin{cor}\label{cor:shells or spurs in a complex}
Let $X$ be bicollapsible. 
Then any near-immersion $E\rightarrow \X $ with $E$ compact and $\pi_1E=1$ has the property that either $E$ consists of a single 2-cell whose boundary is an embedded cycle,
or $E$ has at least two shells and/or spurs.
\end{cor}

\begin{proof}
By Lemma~\ref{lem:towers exist},
let $E\rightarrow Y$
be a maximal tower lift of $E\rightarrow \X \rightarrow \widetilde X$.
Note that $Y$ is compact, $\pi_1Y=1$, and $Y\rightarrow \widetilde X$ is an immersion.
By bicollapsibility, $Y$ is either trivial, 
or $Y$ has two (or more) 2-cells that collapse along free faces.
Observe that preimages of collapses in $Y$ are collapses in $E$.
If $Y$ is a 0-cell or 1-cell, then so is $E$.
If $Y$ is a single 2-cell that collapses along a free face, then the preimage of this 2-cell in $E$ either has two cells that both collapse, or $E$ has a single 2-cell, whose boundary is embedded by Corollary~\ref{cor:2-cells embed in Xtilde}, that collapses, together with a graph. 
Since $\pi_1E=1$, the graph is empty or provides a spur.
Finally, if $Y$ has two collapses, then their preimages provide two collapses, and we conclude that $E$ has at least two cutcells and/or shells and/or spurs.
The result therefore follows from Lemma~\ref{lem:complex dehn property}.
\end{proof}

\begin{rem}\label{rem:tiny innerpath}
Within the context of Corollary~\ref{cor:shells or spurs in a complex},
each 2-cell $R$ is associated with a degree~$n\geq 2$,
and 
it is then natural to redefine the notion of \emph{shell}
so that $|S|<\frac{1}{n}|\boundary_\p R|$.
The conclusion of Corollary~\ref{cor:shells or spurs in a complex} holds with this revised definition.
This is because lobes are attached to $R$ along part of a  component of $\boundary_\p R -\integers_ne$ for some edge $e$.
\end{rem}

We now describe a natural collection of subspaces of $\X$.

\begin{defn}
Let $r$ be a 2-cell of $\X$, let $e$ be an edge of $\partial_\p r=w^n$, and let $e'$ be the edge in the barycentric subdivision of $\X$ from the center of $r$ to the center of $e$.
A \emph{spider} in $r$ at $e$ is the tree consisting of the $\mathbb{Z}_n$-orbit of $e'$. We call the $\mathbb{Z}_n$-orbit of the center of $e$ the \emph{feet} of the spider.

Consider the graph $\Gamma$ obtained from the disjoint union of all spiders of $\X$ by identifying spiders along their feet. 
Observe that there is a natural immersion $\Gamma \rightarrow \X$.

\end{defn}

\begin{figure}
  \centering
  \includegraphics[width=.2\textwidth]{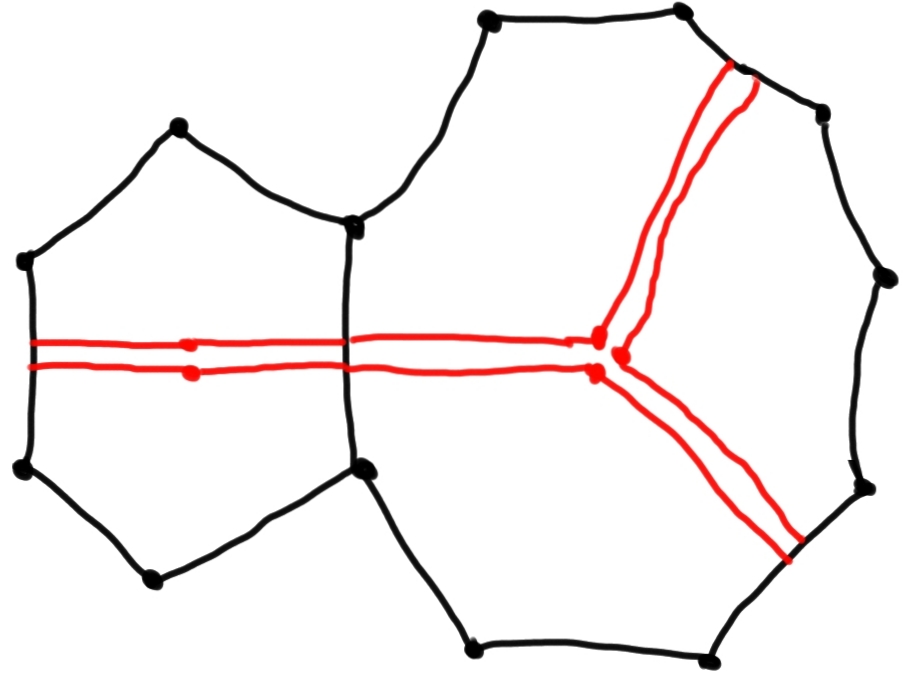}
  \caption{\label{fig:TwoPerEdge} Parts of three natural walls.}
\end{figure}

The images of components of $\Gamma$ in $\X$ are called \emph{divisive trees}.
The \emph{natural walls} of $\X$ are the components of regular neighborhoods of divisive trees.
Note that the natural walls immerse inside the divisive tree. 
See Figure \ref{fig:TwoPerEdge}.

In fact, divisive trees are called trees for a reason:

\begin{lem}
\label{bicollapsible implies tree}
Suppose that $X$ is bicollapsible.
\begin{enumerate}
\item Each component of $\Gamma$ is a tree.
\item Each component of $\Gamma$ embeds in $\X$.
\end{enumerate}
\end{lem}

\begin{rem}
In the barycentric subdivision of $X$, consider a 1-cell $s$, and let $S(s)$ be the union of closed edges with an endpoint at the new vertex corresponding to $s$.
When the attaching map of each 2-cell of $X$ is an embedding, each $S(s)$ is a tree, and the divisive trees in $\X$ are components of the preimage of $S(s)$ under $\X\rightarrow X$.
\end{rem}

\begin{proof}
Consider a shortest embedded path $\tau\subset \Gamma$ mapping to a nontrivial closed path in $\X$.
Choose an annular diagram $A\rightarrow \X$ carrying $\tau \rightarrow \X$. 
Note that $A\rightarrow \X$ is a near-immersion, since $\tau\subset \Gamma$ is embedded.
Choose a reduced disk diagram $D\rightarrow \X$ such that $\boundary_\p D$ is a closed path 
$P\rightarrow A$ that generates $\pi_1A$.
Observe that $A$ has at most one shell (at the 2-cell corresponding to the initial and terminal point of $\tau$; see Remark~\ref{rem:tiny innerpath}).

Consider $A\cup_P D \to \X$, and let $E$ be a maximal tower lift. Since $\pi_1E=1$, by Corollary~\ref{cor:shells or spurs in a complex}, $E$ is either a single 2-cell with embedded boundary cycle, or $E$ has two spurs or shells.

Preimages of spurs of $E$ are spurs of $A\cup_P D$, but $A\cup_P D$ evidently has no spurs.

Suppose $E$ has two shells. Preimages of shells of $E$ that lie in $A$ are shells of $A$. 
The 1-cells of the outer path of a preimage shell either lie in $A$, (but $A$ has at most one shell), or they are interior to $D$, in which case $D$ is not reduced.

Suppose $E$ consists of a single 2-cell with embedded boundary cycle. 
The path $\tau$ is immersed, so there are consecutive 2-cells of $A$ mapping to the single embedded 2-cell in $E$, and therefore cannot be reduced. \qedhere

\end{proof}

\begin{rem}
The subgroup $\stab(\T)$ is not always a codimension-1 subgroup. It is  called 
\emph{divisive} in \cite{HruskaWiseAxioms} and the natural walls carry its associated codimension-1 subgroups.
\end{rem}

We close this section by explaining the following:
\begin{thm}
Let $X$ be bicollapsible and compact.
If $\pi_1X$ is infinite then $\pi_1\dot X$ has a codimension-1 subgroup.
\end{thm}
\begin{proof}
We first observe that the 1-skeleton of $\widetilde X$ has a bi-infinite geodesic.
To see this, note that since $\pi_1X$ is infinite, $\widetilde X$ contains pairs of vertices $p_i,q_i$ with
$\dist(p_i,q_i)= i$.
Hence there are geodesics $\gamma_i$ of length $i$ for each $i$.
Hence there is a bi-infinite geodesic $\gamma$ by Koenig's infinity lemma.

Let $e$ be an edge of $\gamma$.
Consider the divisive tree $T_e$ of $\X$ associated to a lift $\dot e$ of $e$.
And let $\dot \gamma$ be a lift of $\gamma$ containing $\dot e$.
Note that $\dot \gamma$ must also be a geodesic in $\X$.
Finally, let $W$ be a  wall of $\X$ that is alongside $T_e$ and passes through  $\dot e$.
Then $W\cap \dot \gamma$ is a singleton, since any intersection of $W$ and $\dot \gamma$ projects to an intersection of $S(e)$ and $\gamma$.
Moreover, $\dot \gamma$ travels deeply within the two halfspaces of $W$. This last statement holds since $\dot \gamma$ projects to 
$\gamma$ in $\widetilde X$, but the two rays of $\gamma$ get arbitrarily far from $e$.
Since $\gamma$ gets far from the projection of $W$ in $\widetilde X$ (contained in $S(e)$), $\dot \gamma$ gets far from $W$ in $\X$.

Finally, this demonstrates that the subgroup of the stabilizer of $T_e$ that stabilizes $W$ is a codimension-1 subgroup: Indeed, the two ends of $\dot\gamma$ project to distinct ends of $H\backslash \X$, hence the corresponding rays yield distinct ends of the coset diagram $H\backslash \pi_1\dot X$.
\end{proof}

\section{Branched covers and CAT(0) 2-complexes}

While the examples indicated in \S\ref{examples sec} are extensive, geometric properties of the complex $\X$ are often more naturally explored without bicollapsibility.
Two examples follow.

\begin{thm}\label{thm:typical branched}
Let $X$ be a 2-complex with the following properties:
\begin{enumerate}
\item Each 0-cell $v$ has $\girth(\link(v))\geq p$.
\item Each 2-cell $r$ has $|\boundary_\p r|\geq q$.
\item\label{pq ineq} We have $ \frac 2p + \frac 1q \le 1$.
\end{enumerate}
Then the complex $\X$ admits a CAT(0) metric.
If the inequality in \eqref{pq ineq} above is strict, $\dot X$ admits a metric of negative curvature.
Each natural immersed wall is embedded and 2-sided (hence an actual wall) and $\pi_1\dot X$ acts properly on the CAT(0) cube complex associated to the natural wall structure on $\X$.
Moreover, if $X$ is compact then each natural wall has f.g.\ and quasi-isometrically embedded stabilizer.
 (Hence if $\pi_1X$ is word-hyperbolic, then $\pi_1X$ acts cocompactly.) 
\end{thm}

\begin{rem}[Cubulation via Small Cancellation]
Note that in the above setting with $q\ge 3$ all polygons of $\widetilde X$ have at least 3 sides,
and all pieces have length at most 1.
Consequently, $\X$ satisfies the $B(6)$ condition \cite{WiseSmallCanCube04}.
It has a different wallspace structure, also preserved by the group action, than the one we are considering above, 
in which antipodal edges are paired (after subdividing).
Therefore $\pi_1\dot X$ acts properly on the dual cube complex,
and moreover, it acts properly and cocompactly when $\dot X$ is compact.
\end{rem}

Note that the scenario of Theorem~\ref{thm:typical branched} holds
when $X$ is a nonpositively curved piecewise Euclidean 2-complex whose 2-cells
have angles $<\pi$.

\begin{const}[Triangle Subdivision]
Following the setup of Theorem~\ref{thm:typical branched}:
For each 2-cell $r$ of $X$, and associated $2$-cell $\dot r$ of $\dot X$ with degree $d$,
we subdivide $\dot r$ into $d |\boundary_\p r|$  triangles meeting at a central new vertex.
We introduce a piecewise nonpositively-curved metric to $\dot X$ in which each triangle is isosceles with base side length $1$, as follows:
Let the central angle be given by $\pi/q$, and the base angles by $\pi/p$. Note that inequality \eqref{pq ineq} implies that there is a metric of constant non-positive curvature on the triangle with the given angles, which is moreover negative if \eqref{pq ineq} is strict.

The result on $\X$ is a CAT(0) complex built from metric triangles.
Moreover, each natural tree is convex in the metric,
and intersects central vertices and centers of edges.
Furthermore, the complementary regions are finite neighborhoods of original 0-cells.
\end{const}
\begin{proof}[Proof of Theorem~\ref{thm:typical branched}]
To see that the subdivision is CAT(0) with triangles as given, by \cite{Gromov87,BridsonHaefliger}
it suffices to verify
that the length of any cycle in $\link(v)$ is at least $2\pi$.
When $v$ is a new 0-cell, its link is a cycle of length $|\boundary_\p \dot r|\cdot \frac \pi q = d|\boundary_\p r| \frac \pi q \ge 2\pi$.
When $v$ is an old 0-cell, the girth in the subdivision is twice the original girth,
and hence at least $2\cdot p$ by hypothesis.
Because each angle at an old 0-cell is $\pi/p$, the length of a cycle in $\link(v)$ is at least $2p\cdot \frac \pi p =2\pi$.

The walls are convex since their segments meet edges orthogonally
(hence locally isometric to the inclusion
 $\{\frac12\}\times T \subset [0,1]\times T$ where $T$ is tree)
and meet the central vertices with an angle $ |\boundary_\p r| \cdot \frac \pi q \geq \pi$ on each side.
Hence they are graphs that immerse by a local isometry, and are hence convex. 

For properness, note that any infinite order element of $\pi_1 \dot X$ has an axis which is cut by a wall (see \cite{HruskaWiseAxioms}).
\end{proof}
\begin{figure}
  \centering
  \includegraphics[width=.35\textwidth]{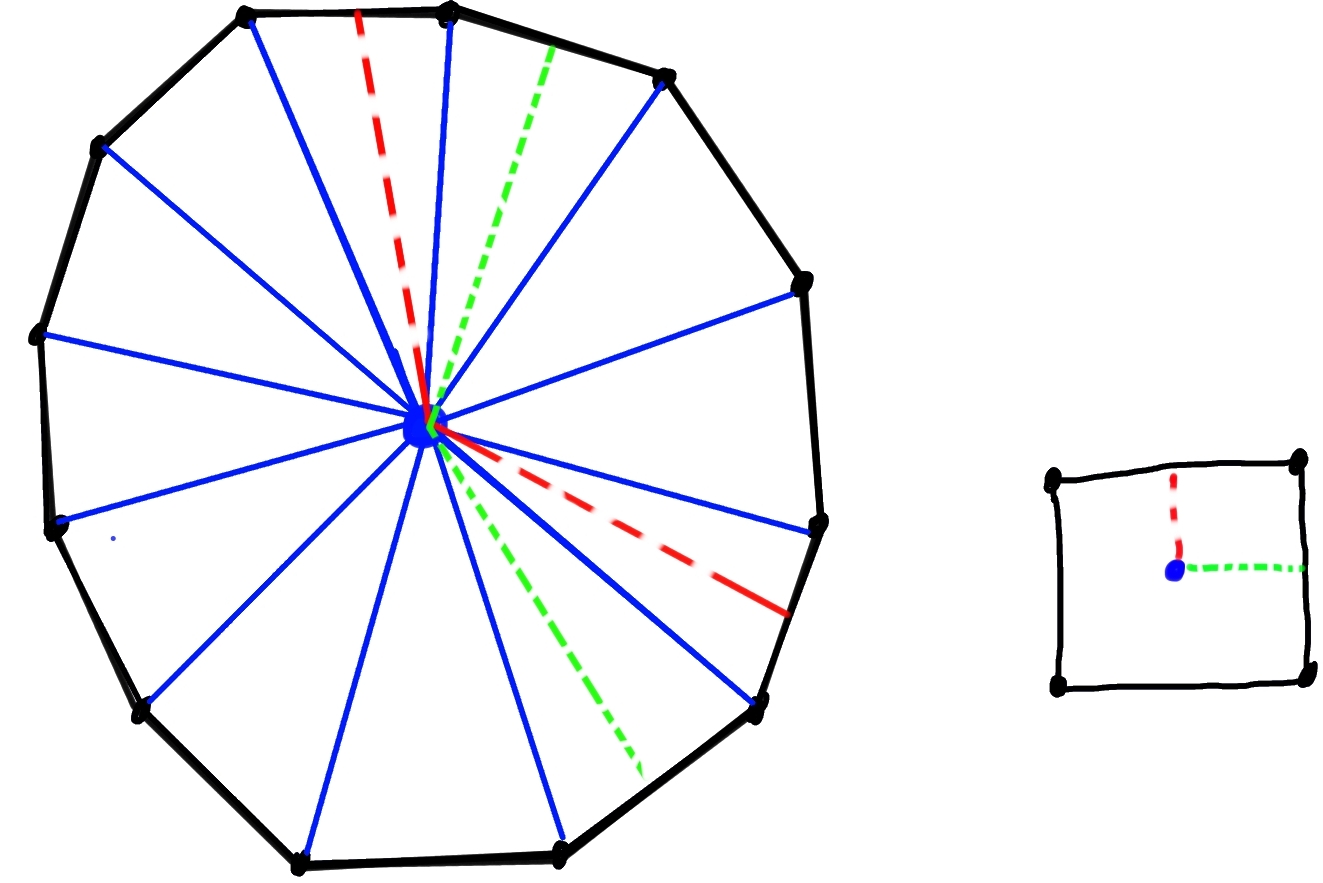}
  \caption{\label{fig:NaturalWalls} Parts of two natural walls passing through a new vertex}
\end{figure}

In general,  $\dot X$ may not be aspherical and $\pi_1\dot X$ may not be  word-hyperbolic. Indeed:
\begin{exmp}
A branched cover of a 2-sphere with two degree~$n$ branched points gives another 2-sphere.
A branched cover with three degree~$2$ branched points gives a torus.
\end{exmp}
Both these examples are linked to spheres in $X$, and the pathology in $\dot X$ will occur if there is a small 2-sphere in $\widetilde X$. However:
\begin{prob}
Suppose $X$ is aspherical. Is $\dot X$ atoroidal? Is $\pi_1\dot X$ aspherical? 
 Hyperbolic? 
\end{prob}

\begin{rem}[Impact on Asphericity]
It is unclear how  $\pi_2\dot X$ is impacted by the asphericity of $X$.
(Of course, without collapsing redundant 2-cells in $\X$, one obtains some essential spheres for each 2-cell of $\dot X$.)

Diagrammatic reducibility is a strong version of asphericity. 
We refer to \cite{Gersten87b} for detail and a connection to angle assignments.
A small step supporting the impact of asphericity of $X$ on $\dot X$ is 
that $\dot X$ is also diagrammatically reducible.
Indeed, any near-immersion to $\dot X$ projects to a near-immersion in $X$.
\end{rem}

\section{Bislim implies Bicollapsible}\label{sec:bislim bicollapse}
Given a preorder $\preceq$, the notation $a\prec b$ means $a\preceq b$ but $b\not\preceq a$.

\begin{defn}[bislim]\label{defn:bislim}
The 2-complex $X$ is \emph{bislim} if there is 
a $\pi_1X$-invariant preorder $\preceq$ on (significant) 1-cells of $\widetilde X$
such that the following conditions are satisfied:
\begin{enumerate}
\item Each 2-cell $r$ has associated (significant) 1-cells  $r^+$ and  $r^-$.
\item $\boundary_\p r$ traverses $r^+$ exactly once.
\item If $x^+ \subset \boundary y$ with $y\ne x$ then $x^+ \prec y^+$.
\item\label{bislim:4}  If $\boundary x \supset y^-$ with $y \ne x$ then $y^+ \prec x^+$.
\end{enumerate}
\end{defn}

Note that \eqref{bislim:4} implies that $r^+$ and $r^-$ are distinct whenever $r$ has no free faces.

\begin{rem}\label{rem:good stacking is bislim}
The notion of \emph{good stacking} was introduced by
Louder and Wilton in \cite{LouderWilton2014}. 
It is shown in \cite{BambergerCarrierGasterWise2018} that a 2-complex with a good stacking 
has a bislim structure in the above sense.

Definition~\ref{defn:bislim} is slightly more general than the original definition in \cite{HelferWise2015}, but the conclusions obtained for bislim complexes are the same.
We refer to \cite{BambergerCarrierGasterWise2018} for the details.
\end{rem}

\begin{prob}
Suppose that $\widetilde X$ is 2-collapsing, and no 2-cell of $\widetilde X$ is a proper power.
Is $\widetilde X$ bislim (with the trivial group action)?
Does $\widetilde X$ have a good stacking?
\end{prob}

 \begin{lem}\label{lem:bi bi}
 Let $X$ be bislim, and suppose that $r^-$ is traversed exactly once by $\boundary_\p r$ for each 2-cell $r$.
 Then $X$ is bicollapsible.
 \end{lem}

\begin{proof}
We verify the hypothesis of Lemma~\ref{lem:bicollapse criterion}.
Let $Y\subset \widetilde X$ be a compact subcomplex.
If $Y$ is a graph we are done.
If $Y$ has a single 2-cell $r$, then $r$ collapses along $r^+$ and we are done.
Now suppose $Y$ has at least two 2-cells.

Suppose there exists a 2-cell $r$ such that $r^+$ is maximal in the sense that
there does not exist a 2-cell $s$ with $r^+\prec s^+$.
Then $r$ collapses along $r^+$.
Indeed, $\boundary_\p r$ traverses $r^+$ once, and if $s$ is a 2-cell with $r^+$ in $\boundary s$ then $r^+\prec s^+$.

Suppose  there exists a 2-cell $r$ that is minimal in the sense that:
 there does not exist a 2-cell $s$ with $s^+\prec r^+$.
 Then $r$ collapses along $r^-$.
Indeed, here we use the additional hypothesis that 
$r^-$ is traversed exactly once by $\boundary_\p r$.
If $s$ is a 2-cell with $r^-$ in $\boundary s$, then $s^+\prec r^+$.

As $\prec$ is a partial order when restricted to $\{r^+:r\in 2$-cells$(Y)\}$, maximal and minimal elements in the former sense exist. 
Because $Y$ has at least two 2-cells, it has at least two extremal cells, each of which collapses as above.
\end{proof}
 \begin{exmp} The proof of Lemma~\ref{lem:bi bi} uses that $\boundary_\p r$ traverses $r^-$ exactly once. The following  example shows that extra hypothesis is necessary.
 Consider the presentation: $\langle a,b,c \mid aa^{-1}b, bc \rangle$.
 Its universal cover consists of an $a$-line, with $b$ and $c$ loops at each vertex.
 We declare $a\prec b \prec c$, which induces a preorder on the 1-cells of $\widetilde X$. It is immediate to check the conditions for bislimness, however $X$ is not bicollapsible: the concatenation of the 2-cells $aa^{-1}b$ and $bc$ has two cells but only one collapse.
\end{exmp}

\section{3-collapsibility and Cubulation}
In this section we work under the assumption that $X$ is $3$-collapsing. 

\begin{conv}
Following section \S\ref{sec:branched sec}, we again fix the convention that boundary cycles of 2-cells of $X$ are immersed.
\end{conv}

A \emph{ladder} $L$ is a disk diagram that is the union of a finite sequence
of closed 1-cells and 2-cells $C_1,\ldots, C_n$ where $n\geq 2$
and $C_i\cap C_j=\emptyset$ when $j>i+1$.
Equivalently, for $1<i<n$, each $C_i$ is a cutcell with two lobes.

\begin{thm}\label{thm:123}
 For every near-immersion $D\rightarrow \X$ with $D$ compact and contractible, either:
\begin{enumerate}
\item $D$ is a single 0-cell or 2-cell,
\item $D$ is a ladder,
\item $D$ has at least three shells and/or spurs.
\end{enumerate}
\end{thm}
As in Remark~\ref{rem:tiny innerpath}, each shell $R$ above has innerpath $|S|<\frac1n|\boundary_\p R|$ where $n=\degree(R)$.
\begin{proof}
The proof is by induction.

Let $Y$ be the image of the lift of $D\rightarrow \dot X \rightarrow X$ to $\widetilde X$.

Suppose $Y$ has a single 2-cell $R$ which collapses along a 1-cell $e$.
Then each 2-cell of $D$ is a shell or cutcell.
It is easy to see that the conclusion holds in that case.

Suppose $Y$ has two 2-cells $R_1,R_2$. Then they  collapse along 1-cells
$e_1,e_2$. Hence the same reasoning holds as in the previous case.

Suppose $Y$ has three or more 2-cells, then $Y$ has 2-cells
$R_1,R_2,R_3$ collapsing along 1-cells $e_1,e_2,e_3$.
Each 2-cell $C$ mapping to some $R_i$ is a cutcell or shell of $D$.
Note that some $C_i$ maps to $R_i$ for each $i$.
If each $C_i$ is a shell, we are done as $D$ has three shells.
Otherwise, at least one $C_i$ is a cutcell.

If some $C_i$ has three lobes, then each lobe contains a shell (beyond $C_i$) by induction, so $D$ has three shells.

If some $C_i$ has two lobes, then each of these is either ladder or contains two shells (beyond $C_i$) and hence $D$ is either a ladder, as it is built by combining two ladders along $C_i$, or $D$ has three shells.
\end{proof}

For a subset $W\subset \X$, let $N(W)\subset \X$ denote the \emph{carrier} of $W$, i.e.~the smallest subcomplex of $\X$ containing $W$. 

Recall the graph $\Gamma\rightarrow \X$, a union of divisive trees, whose images determine natural walls in $\X$ (see Figure~\ref{fig:TwoPerEdge}). Note that each vertex of $\Gamma$ maps to 
the interior of either a 2-cell or a 1-cell; we call the former \emph{face} vertices of $\Gamma$.

Consider an arc $J \subset \Gamma$ between face vertices of $\Gamma$. Let $M(J)$ be the 2-complex consisting of the union of a closed 2-cell for each face vertex of $J$, such that 2-cells associated to semi-consecutive vertices meet along an arc in their boundaries corresponding to the maximal piece between them, containing the intervening edge vertex. Observe that there is an immersion $M(J)\rightarrow \X$. 

\begin{lem}\label{lem:M(J) ladder}
$M(J)$ is a ladder.
\end{lem}

\begin{proof}
Observe that, in the construction of $M(J)$, the two maximal pieces in the boundary of a cell cannot overlap; if they did one would find in $\widetilde X$ a 2-cell whose boundary is a concatenation of two pieces, which is impossible by Lemma~\ref{lem:2cell intersections}. 
This implies that $M(J)$ is a ladder by definition.
\end{proof}

\begin{lem}\label{lem:M(J) convex}
$M(J)\rightarrow \X$ is injective, and its image is a convex subcomplex $N(J)\subset \X$.
\end{lem}

\begin{proof}
Let $\gamma\rightarrow \X$ be a geodesic that is path homotopic to a path $\sigma\rightarrow M(J)\rightarrow \X$.
Let $D$ be a disk diagram between $\gamma$ and $\sigma \rightarrow \X$,
and suppose $(D,\sigma)$ are chosen so that $D$ has minimal area with this property.
Let $Y=D\cup_\sigma M(J)$, and let $Y\rightarrow T$ be a maximal tower lift of $Y\rightarrow \X$.

Observe that the only possible 2-cells of $T$ that are shells are the 2-cells that are images of the first and last 2-cells of $M(J)$.
Indeed, no outer path of a shell lies on $\gamma$, since it is a geodesic, and the remaining 2-cells to consider in $D$ cannot be shells, since they are surrounded by reduced diagrams. 
Indeed, $D$ is reduced and the 2-cells of $D$ along $M(J)$ cannot form a cancellable pair with 2-cells of $M(J)$ by minimality.
Hence at most two 2-cells of $T$ are shells, so $T$ is a ladder by Theorem~\ref{thm:123}.

The map $M(J)\rightarrow T$ must send consecutive 2-cells to consecutive 2-cells, and semi-consecutive 2-cells to semi-consecutive 2-cells. Hence $M(J)$ is a subladder of $T$. 

We now consider the path $\gamma \rightarrow T$ whose endpoints lie on $M(J)$. If $\gamma$ were not contained in $M(J)$, then the outermost 2-cell of $T$ that it wraps around shows that $\gamma$ is not a geodesic, since a piece is less than half the perimeter of a 2-cell by Lemma~\ref{lem:2cell intersections}. Therefore $\gamma\subset M(J)$. It follows that $M(J)\rightarrow \X$ is injective, and the image is convex.
\end{proof}

\begin{lem}\label{lem:convex carriers}
Let $W$ be a wall in $\X$.
Then $N(W)$ is convex.
\end{lem}

\begin{proof}
This is an immediate consequence of Lemma~\ref{lem:M(J) convex}: any two points in $N(W)$ lie in the image of $M(J)$, for some arc $J\subset \Gamma$.
\end{proof}

\begin{rem}The same proof shows that $N(T)$ is convex for each divisive tree $T$.\end{rem}

\begin{cor}\label{cor:braided bigons}
Let $J_1\subset T_1$ and $J_2\subset T_2$ be arcs in divisive trees that start and end on face vertices.
Suppose  $J_1$ and $ J_2$ have the same endpoints.
Then $N(J_1)=N(J_2)$.
Hence their pairs of dual edges lie in a piece.
\end{cor}
\begin{proof}
Otherwise, choosing a minimal disk diagram whose boundary path travels ``alongside'' $J_1$ and $J_2$, 
we could form a complex from $N(J_1)\cup N(J_2) \cup D$
that isn't a ladder, but has only two shells. Its image in $\X$ would violate
Theorem~\ref{thm:123}.
\end{proof}

\begin{lem}\label{lem:equisector}
For each edge $x$ of a 2-cell $R$ of degree~$n$,
there is an edge $y$, such that the associated divisive trees $T_x, T_y$
have no common vertex besides the vertex dual to $R$.
\end{lem}

\begin{proof}
By Corollary~\ref{lem:2cell intersections}, each pair of
 2-cells have connected intersection.
Hence no 2-cell in $\widetilde X$ has boundary cycle that is the concatenation of two pieces. \end{proof}

Note that this is a consequence of a slightly weaker condition than $3$-collapsible.

We refer to \cite[Lem~7.16]{WiseCBMS2012} for the following properness criterion:
\begin{prop}\label{prop:cut axis}
Let $\widetilde Y$ be a graph that is a wallspace, with an action by a group $G$.
Suppose that for each infinite order $g\in G$, there is a $g$-invariant embedded line $R\subset \widetilde Y$ and a wall $W$ that separates $R$ into rays.
Then $G$ acts with torsion-stabilizers on the dual.
\end{prop}
We verify the above condition by showing that geodesics are cut at a single point,
hence a $g$-invariant quasi-geodesic axis $R$ is  cut as well.

The following cocompactness criterion
 is a generalized restatement of Sageev's result \cite{Sageev97},
 and we refer to  \cite{HruskaWiseAxioms} for a more elaborate discussion.
\begin{prop}\label{prop:cocompactness}
Let $G$ act properly and cocompactly on a graph $\widetilde Y$ that is also a wallspace.
Suppose the stabilizer of each wall is a quasiconvex subgroup of $G$.
Then $G$ acts cocompactly on the dual cube complex.
\end{prop}

\begin{thm}
Let $X$ be 3-collapsing, and suppose that the attaching maps of 2-cells are immersions.
Then $\pi_1\dot X$ acts properly (and cocompactly) on the CAT(0) cube complex dual to the wallspaces of $\X$.
\end{thm}
\begin{proof}
Below we will verify the hypothesis of Proposition~\ref{prop:cut axis} with 
$\widetilde Y$ equal to the 1-skeleton of $\X$.
Since all torsion-subgroups of the word-hyperbolic group $G=\pi_1\dot X$ are finite \cite{ABC91},
we see that $G$ acts with finite stabilizers. 
Cocompactness holds by Proposition~\ref{prop:cocompactness} and Lemma~\ref{lem:convex carriers}.
Consequently, the action is proper since it is both cocompact and has finite stabilizers.

Let $\gamma$ be a geodesic in $\X$.
Let $\epsilon_1$ be an edge of $\gamma$.
Let $W$ be a wall dual to $\epsilon_1$.

If $W\cap \gamma$ is a singleton, we are done.
Otherwise, suppose $W$ intersects $\gamma$ at a second edge $\epsilon_2$.
By Lemma~\ref{lem:convex carriers}, there is an arc $J_W\subset W$ whose endpoints are on $\epsilon_1,\epsilon_2$,
such that a subpath $\epsilon_1 \gamma'\epsilon_2\subset \gamma$ lies in $N(J_W)$.

It will be convenient to decompose the argument into three cases according to the number of 2-cells in the carrier $N(J_W)$.
This corresponds to the length of $J_W$ in a metric such that each edge of $W$ has length $\frac 12$.

{\bf Case $|J_W|\ge3$:}
If $|J_W|\geq 3$, then $N(J_W)$ is a ladder consisting of at least three consecutive 2-cells.
Let $R$ be an intermediate 2-cell of $N(J_W)$, and let $R_e$ and $R_w$ be the 2-cells in $N(J)$ on either side of $R$. 
Let  the cycle of $\boundary R$ be of the form $nesw$ where $e$ and $w$ are maximal
pieces of $R$ with $R_e$ and $R_w$.
We refer to Figure~\ref{fig:Case1and2} on the left.

By Lemma~\ref{lem:equisector}, there is a wall $V$ dual to 
 edges of $\gamma\cap R$ so that $V\cap R_e= V\cap R_w = \emptyset$.

We verify that $V$ cannot cut through $L$ except at $V\cap R$.
Indeed, suppose $V$ intersects $L$ at a 2-cell $R_w'$, which is on the $R_w$ side of $R$.
Corollary~\ref{cor:braided bigons} applied to the arc $J_W$ in $W$  joining the midpoints of $R_w'$ and $R$,
and  the arc $J_V$ in $V$ joining the midpoints of $R_w'$ and $R$,
shows that the carriers of these arcs are the same.
However, the carrier of $J_W$ is the subladder from $R_w'$ to $R$,
but the carrier of $J_V$ contains  an additional 2-cell carrying the  additional edge where $V$ exits $R$ at either $n$ or $s$.
The analogous argument holds for a 2-cell $R_e'$ on the $B$ side.

We now show that $V$ is not crossed by $\gamma_1$ or $\gamma_2$ beyond $L$.
Suppose, that $\gamma_1$ crosses $V$ beyond $L$, then $\epsilon_1$ lies on a 2-cell $C$ of $N(V)$ by Lemma~\ref{lem:convex carriers}. Let $W'$ be the arc in $W$ from the center of $R$
to the center of $C$. Let $V'$ be the arc in $V$ from the center of $R$ to the center of $C$.
Then we obtain a contradiction precisely as in the previous case.
The analogous argument works when $\gamma_2$ crosses $V$ beyond $L$.

{\bf Case $|J_W|=1$:}
If two edges $\epsilon_1,\epsilon_2$ of $\gamma$  are dual to a wall $W$ and lie in the same 2-cell $R$, then we can assume $\epsilon_1$ is leftmost on $\gamma \cap R$ with this property.
By Lemma~\ref{lem:equisector}, there is a wall $V$ cutting $R$ and separating $\epsilon_1,\epsilon_2$. 
Either $V$ isn't dual to any other edge of $\gamma$, or $V$ cuts
$\gamma$ in a pair of 2-cells on a ladder of length~$\geq 3$ (dealt with in Case~3),
or $V$ cuts $\gamma$ in a pair of 2-cells on a ladder of length~$2$ (dealt with in Case~2).
See the second diagram in Figure~\ref{fig:Case1and2}.

{\bf Case $|J_W|=2$:}
Suppose  $N(J_W)$ consists of two 2-cells.
Lemma~\ref{lem:equisector} provides a tree $T_V$ intersecting $T_W$ at the center of $R$
and at no other point. The 2-cell $R$ has two edges dual to $W$, namely $\epsilon_2$
and another edge $\epsilon_w$. Moving in $\boundary_\p R$ in the direction from $\gamma\cap R$ to $\gamma\cap W$, we choose the first edge $\epsilon_v$ dual to $T_V$ arising after $\epsilon_w$.
Let $V$ be the wall dual to $\epsilon_v$, so that $V$ separates $\epsilon_w,\epsilon_2$.

If $V$  is not crossed by $\gamma$ at another edge then we are done.
Otherwise, note that $V$ cannot cross $\gamma$ on the left, since then Lemma~\ref{lem:convex carriers} would imply that $\epsilon_1$ lies on the carrier of $V$ and Corollary~\ref{cor:braided bigons} would then imply that $\epsilon_v$ and $\epsilon_w$ lie in a common piece. 
See the upper middle diagram of Figure~\ref{fig:Case1and2}.

If $V$ returns using an arc $J_V$ of length~$\geq 3$ then we are done (as in Case~3).
Hence we assume $J_V$ has length~$2$.
Observe that the edges of $V$ and $W$ in $R$ that are associated to the  other 2-cells of $J_V$ and $J_W$ do not lie in adjacent  pieces of the 2-cell $R$ - or else we would have a non-ladder  with  three 2-cells but only two shells,  as  $R$ cannot be a shell on the geodesic $\gamma$.

Let $U$ be a wall dual to an edge of $\boundary R$ that does not lie in a piece with either.
Then $U$ cannot cross $\gamma$ a second time. There are eight cases to consider, and we exclude them one-by-one, referring to Figure~\ref{fig:ThirteenthImam} for a visual.

\begin{figure}
  \centering
  \includegraphics[width=\textwidth]{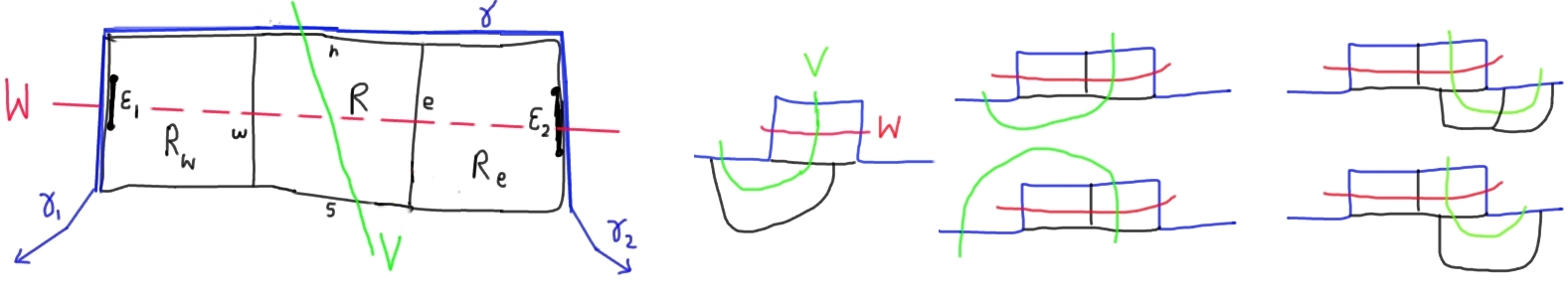}
  \caption{\label{fig:Case1and2} Case 3 on the left.
  Case 1 and the beginning of Case 2 on the right.}
\end{figure}

\begin{figure}
  \centering
  \includegraphics[width=.5\textwidth]{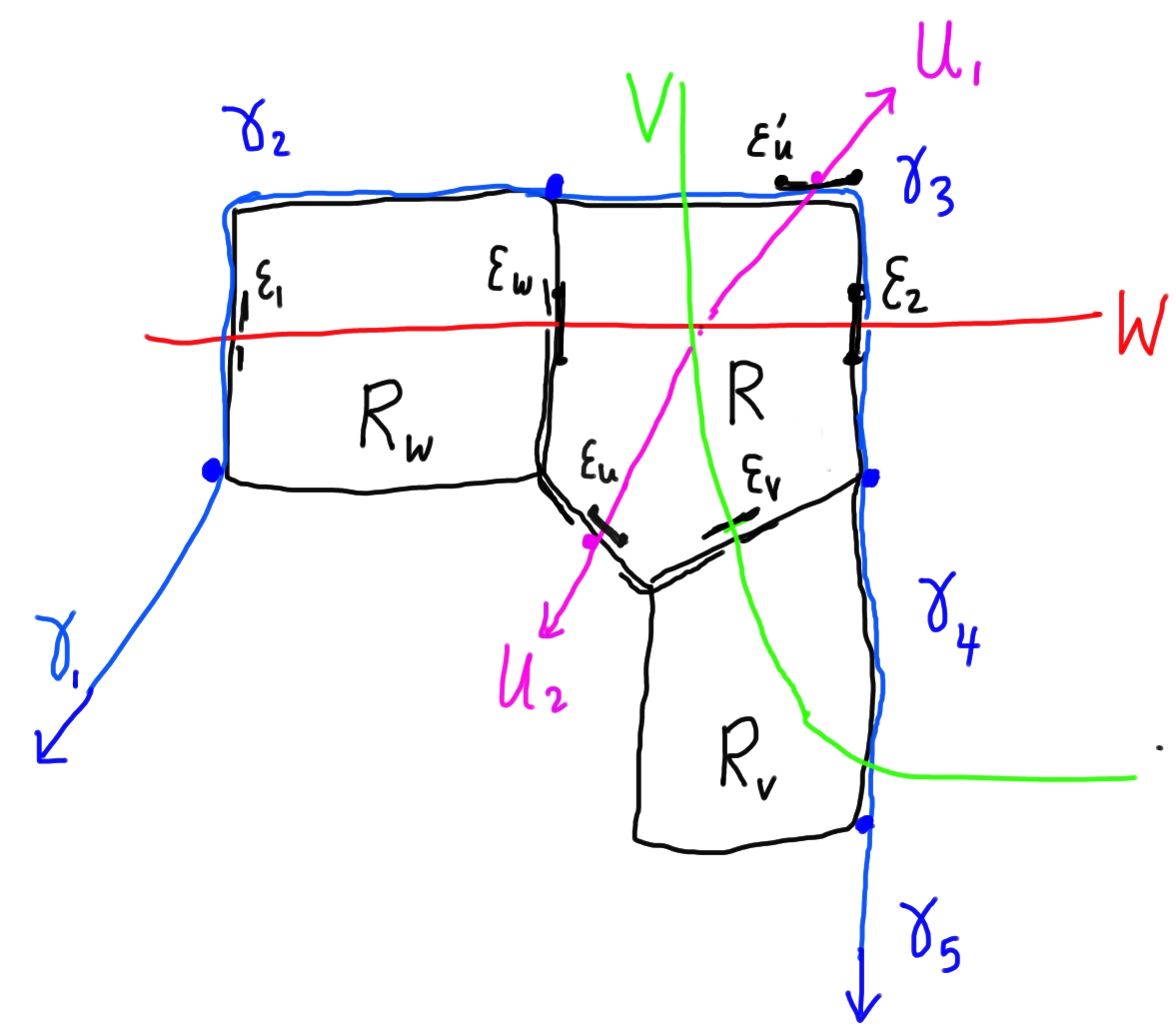}
  \caption{\label{fig:ThirteenthImam} There are 10 cases to exclude,
  according to which of the two sides of $U$ is crossed by which of the 5 subgeodesics of $\gamma$.}
\end{figure}

$U$ does not intersect $\gamma_3$ at a second point,
since $U\cap \boundary R$ consists of exactly two points, and one of them is not in $\gamma$.

$U_2$ does not intersect $\gamma_2$ since  then
by Corollary~\ref{cor:braided bigons},
$\epsilon_w$ and $\epsilon_u$ lie in the same piece,
which is impossible.
Similarly $U_2$ does not intersect $\gamma_4$, since 
then
$\epsilon_v$ and $\epsilon_u$ lie in the same piece.

$U_2$ does not intersect $\gamma_1$ since then $\epsilon_1$
would lie on $N(U)$ by Lemma~\ref{lem:convex carriers}, 
and so $J_W\subset N(U)$, and so $\epsilon_w,\epsilon_u$ lie in the same piece by Corollary~\ref{cor:braided bigons}.
Similarly $U_2$ does not intersect $\gamma_5$.

Note that $\epsilon_u'$ and $\epsilon_w$ cannot lie in the same piece,
since they are separated by $V$ which is not dual to an edge in the same piece
with $\epsilon_w$.
Consequently, we can repeat 
each of the arguments above verbatim, replacing $U_2$ with $U_1$ and $\epsilon_u$ with $\epsilon_u'$.
We conclude that $U_1$ does not intersect $\gamma_1$, $\gamma_2$, $\gamma_4$, or $\gamma_5$, as claimed.
\end{proof}

{\bf Acknowledgment:} We are grateful to Tim Riley for helpful comments
on the generalized deHn property.

\bibliographystyle{alpha}
\newcommand{\etalchar}[1]{$^{#1}$}
\def\cprime{$'$} \def\polhk#1{\setbox0=\hbox{#1}{\ooalign{\hidewidth
  \lower1.5ex\hbox{`}\hidewidth\crcr\unhbox0}}} \def\cprime{$'$}
  \def\cprime{$'$} \def\polhk#1{\setbox0=\hbox{#1}{\ooalign{\hidewidth
  \lower1.5ex\hbox{`}\hidewidth\crcr\unhbox0}}}

\end{document}